\documentclass[a4, 12pt]{amsart}
\usepackage{amssymb}
\usepackage{amstext}
\usepackage{amsmath}
\usepackage{amscd}
\usepackage{latexsym}
\usepackage{amsfonts}
\usepackage[all]{xy}	%%the package to write commutative diagram

\usepackage{graphicx}

\newtheorem{Theorem}{Theorem}[section]

\newtheorem{theorem}[Theorem]{Theorem}
\newtheorem{prop}[Theorem]{Proposition}
\newtheorem{pro}[Theorem]{Proposition}

\newtheorem{proposition}[Theorem]{Proposition}
\newtheorem{lem}[Theorem]{Lemma}

\newtheorem{lemma}[Theorem]{Lemma}

\newtheorem{cor}[Theorem]{Corollary}

\newtheorem{dfn}[Theorem]{Definition}
\newtheorem{setting}[Theorem]{Settings}

\def\x{\underline x}

\def\Hom{\operatorname{Hom}}

\def\Ext{\operatorname{Ext}}

\newcommand{\rmg}{\mathrm{g}}

\newcommand{\calD}{\mathcal{D}}

\newcommand{\calF}{\mathcal{F}}

\newcommand{\calM}{\mathcal{M}}
\newcommand{\calN}{\mathcal{N}}

\newcommand{\fka}{\mathfrak{a}}
\newcommand{\fkb}{\mathfrak{b}}

\newcommand{\fkm}{\mathfrak{m}}
\newcommand{\fkn}{\mathfrak{n}}

\newcommand{\fkp}{\mathfrak{p}}
\newcommand{\fkq}{\mathfrak{q}}

\newcommand{\q}{\mathfrak{q}}
\newcommand{\m}{\mathfrak{m}}

\def\H{\operatorname{H}}

\def\depth{\operatorname{depth}}
\def\Supp{\operatorname{Supp}}
\def\Ann{\operatorname{Ann}}
\def\Ass{\operatorname{Ass}}
\def\Assh{\operatorname{Assh}}
\def\Min{\operatorname{Min}}

\def\height{\operatorname{ht}}

\def\Spec{\operatorname{Spec}}

\begin{document}

\title{eventually index of reducibility  on sequentially Cohen-Macaulay modules }

\author{Hoang Le Truong}
\address{Institute of Mathematics, VAST, 18 Hoang Quoc Viet street, 10307, Hanoi, Viet Nam}
\email{hltruong@math.ac.vn}
\thanks{ 
This research is funded by Vietnam National Foundation for Science
and Technology Development (NAFOSTED) under grant number
101.04-2014.15.
\endgraf
\endgraf
}

\maketitle

\begin{abstract}
It is shown that  a module is sequentially Cohen-Macaulay  if and only if the index of reducibility for distinguished parameter ideals are eventually constant with special value. As corollaries to the
main theorem we given to characterize the Gorensteinness, Cohen-Macaulayness of local rings in term of eventually index of reducibility for distinguished parameter ideals.
\end{abstract}

\section{Introduction} Throughout this paper let  $R$  be a commutative Noetherian local  ring with  maximal ideal $\m$. Let $M$ be a finitely generated $R$-module of dimension $d >0$. Then we say that an $R$-submodule $N$ of $M$ is {\it irreducible}, if $N$ is not written as the intersection of two larger $R$-submodules of $M$. Every $R$-submodule $N$ of $M$ can be expressed as an irredundant intersection of irreducible $R$-submodules of $M$ and the number of irreducible $R$-submodules appearing in such an expression depends only on $N$ and not on the expression. Let us call, for each parameter ideal $\q$ of $M$, the number $\calN(\q;M)$ of irreducible $R$-submodules of $M$ that appear in an irredundant irreducible decomposition of $\q M$ {\it the index of reducibility} of $M$ with respect to $\q$. Let $S$ be the set of parameter ideals of $M$ and $a$ integer number. Then we say that the index of reducibility for $S$ are {\it eventually} $a$  if there some integer number $n$ such that for all parameter ideals $\q\in S\cap \m^n$, we have $\calN(\q;M)=a.$  

Firstly, perhaps less widely known is a result of Northcott and Rees which states that if
every parameter ideal  of $R$ is irreducible then $R$
is Cohen-Macaulay \cite[Theorem 1]{NR}. Thus, $R$ is Gorenstein if and only if every parameter
ideal is irreducible.  Recently, it was shown by many works that the index of reducibility of parameter ideals can be used to deduce a lot of information on the structure of some classes of modules such as Gorenstein rings(\cite{No},\cite{NR}, \cite{Tr2} \cite{TY}),Cohen-Macaulay modules (\cite{CQT}, \cite{Tr}, \cite{Tr1}, \cite{Tr2}, \cite{TY}), Buchsbaum modules (\cite{GSa}), generalized Cohen-Macaulay modules (\cite{CT}), and so on. In \cite[Theorem 5.2]{CQT}, N. T. Cuong, P. H. Quy and author showed that $M$ is a Cohen-Macaulay module if and only if the index of reducibility for all parameter  ideals of $M$ are eventually Cohen-Macaulay type.    %$r_d(M)$ $\dim_{k}(0):_{\H^d_{\frak m}(M)}\m$ , where $\H^j_\m(M)$ the $j$-th local cohomology of $M$ with support in $\m$.
 The necessary condition of this result can be extended by the author in \cite[Theorem 1.1]{Tr} for a large class of modules called sequentially Cohen-Macaulay modules. %The notion of a sequentially Cohen-Macaulay module was introduced first by Stanley \cite{St} for the graded case and in \cite{Sch} for the local case. 

To sate this result, let us fix some notation. A filtration $$\calD:0\subsetneq D_0\subsetneq D_1 \subsetneq \cdots \subsetneq D_\ell =M $$ of $R$-submodules of $M$ is  called the {\it dimension filtration} of $M$,  if  for all $0 \le i \le \ell-1$, $D_{i-1}$ is the largest submodule of $D_i$ with $\dim_RD_{i-1}<\dim_R D_i$, where $\dim_R(0) = - \infty$ for convention. We say that $M$ is {\it sequentially Cohen-Macaulay}, %in the sense of \cite{CC},
if  $C_i=D_i/D_{i-1}$   is Cohen-Macaulay for all $1 \le i \le \ell$. Let $\x=x_1,x_2, \ldots, x_d$ be a system of parameters of $M$. Then $\x$ is said to be {\it distinguished}, if  $$(x_j \mid d_{i} < j \le d)D_i=(0)$$ for all $0\le i\le \ell$, where $d_i=\dim_R D_i$ (\cite[Definition 2.5]{Sch}). A parameter ideal $\q$ of $M$ is called  {\it distinguished}, if there exists a distinguished system $x_1,x_2, \ldots, x_d$ of parameters of $M$ such that $\q=(x_1,x_2, \ldots, x_d)$. We now denote $r_j(M)=\ell_R((0):_{\H^j_\fkm(M)}\fkm)$ for all $j\in \Bbb Z$ and 
$r(M)=\sum\limits_{j\in\Bbb Z}r_j(M).$ Note that $r_d(M)$ is called {\it Cohen-Macaulay type}.
 With this notation, the author showed in \cite[Theorem 1.1]{Tr} that  if $M$ is a sequentially Cohen-Macaulay module, then the index of reducibility for all distinguished parameter  ideals of $M$ are eventually $r(M)$. From this point of view, it seems now natural to ask  whether the converse of this statement is true. The answer is affirmative, which we are eager to report in the present paper.

\medskip

\begin{theorem}\label{0}
Assume that $R$ is a homomorphic image of a Cohen-Macaulay
local ring. Then the following statements are equivalent.
\begin{enumerate} 
\item[$({\rm 1})$] $M$ is sequentially Cohen-Macaulay.

\item[$({\rm 2})$]  The index of reducibility for all distinguished parameter  ideals of $M$ are eventually $r(M)$.

\item[$({\rm 3})$]  There exists an integer $n$ such that for all distinguished parameter ideals $\q\subseteq \m^n$, we have
$$\calN(\q;M)\le r(M).$$

%\item[$({\rm 3})$]  There exists an integer $n$ such that for all distinguished parameter ideals $\q\subseteq \m^n$, we have
%$$\calN(\q;M)\le\sum\limits_{i=0}^\ell f_0(\q;C_i).$$ 

\end{enumerate}
\end{theorem}

\medskip

From the main result, we get the following results.

\begin{cor}\label{Cmain}
For all  integers $n$  there exists  a distinguished parameter ideal $\frak q\subseteq \frak m^n$, we
    have
$$ r(M)\leqslant \calN(\q;M)$$.
\end{cor}

\begin{cor}
Assume that $R$ is a homomorphic image of a Cohen-Macaulay local ring. Then the following statements are equivalent.
\begin{enumerate} 
\item[$({\rm 1})$] $M$ is Cohen-Macaulay.

\item[$({\rm 2})$]  The index of reducibility for all parameter  ideals of $M$ are eventually Cohen-Macaulay type. 

\item[$({\rm 3})$]  The index of reducibility for all distinguished parameter  ideals of $M$ are eventually Cohen-Macaulay type. 

\end{enumerate}

\end{cor}
\begin{theorem}Assume that $R$ is a homomorphic image of a Cohen-Macaulay local ring. Then
 $R$ is Gorenstein if and only if the index of reducibility for all distinguished parameter  ideals are eventually $1$. 
\end{theorem}

Let us explain how this paper is organized. This paper is divided into 3 sections. We shall prove Theorem \ref{0} and our corollaries  in Section 3. The notion of sequentially Cohen-Macaulay module was introduced  by R. Stanley \cite{St} in graded case, and the local case was studied by \cite{Sch}. Our Theorem \ref{0} partially covers a main result in \cite[Theorem 5.2]{CQT} .
In our argument Goto sequence of type II play an important role. In Section 2 let us briefly note the existence of Goto sequence of type II.

\medskip

\section{Goto sequences}

Throughout this paper let  $R$  be a commutative Noetherian local  ring with  maximal ideal $\m$. Let $M$ be a finitely generated $R$-module of dimension $d >0$. We put $$\Assh_RM=\{\fkp \in \Supp_RM \mid \dim R/\fkp=d\}.$$ Then $$\Assh_RM \subseteq \Min_RM \subseteq \Ass_RM.$$
Let $\Lambda (M)=\{\dim_R L \mid L \ \text{is an}\ R\text{-submodule of}\ M, L \ne (0)\}.$ We then have $$\Lambda (M)=\{\dim R/\fkp \mid \fkp \in\Ass_RM\}.$$ 
We put  $\ell=\sharp \Lambda (M)$ and number the elements  $\{d_i\}_{1\leq i\leq \ell}$ of $\Lambda (M)$  
so that  $$0 \le d_1<d_2<\cdots<d_{\ell}=d.$$
Then because the base ring $R$ is Noetherian, 
for each $1\leq i\leq \ell$ the $R$-module $M$ contains 
the largest $R$-submodule $D_i$ with $\dim_RD_i=d_i$. 
Therefore, letting $D_0=(0)$, we have the filtration 

$$\calD : D_0=(0)\subsetneq D_1\subsetneq D_2\subsetneq \cdots 
\subsetneq D_{\ell}=M$$
of $R$-submodules of $M$, which we call the dimension filtration of $M$. 
The notion of dimension filtration was firstly given 
by P. Schenzel \cite{Sch}. 
Our notion of dimension filtration is a little different from that of \cite{CC, Sch},
but throughout  this paper let us utilize the above definition. It is standard to check that $\{D_j\}_{0\leq j\leq i}$ 
(resp. $\{D_j/D_i\}_{i\leq j\leq \ell}$) 
is the dimension filtration of $D_i$ (resp. $M/D_i$) 
for every $1\leq i\leq \ell$. We put   $C_i=D_i/D_{i-1}$ for $1 \le i \le \ell$.

We note two characterizations of the dimension filtration. Let $$(0)=\bigcap_{\fkp \in \Ass_RM}M(\fkp)$$ be a primary decomposition
of $(0)$ in $M$, where $M(\fkp)$ is an $R$-submodule of $M$ 
with $\Ass_RM/M(\fkp)=\{ \fkp \}$ for each $\fkp \in \Ass_RM$. Then the submodule
$D_{\ell-1} =\bigcap\limits_{\fkp\in\Assh(M)}M(\fkp)$
is called the unmixed component of $M$.
We then have the following.

\begin{pro}[{\cite[Proposition 2.2, Corollary 2.3]{Sch}}]\label{d1}
The following assertions hold true. 
\begin{enumerate}
\item[$(1)$] $D_i=\bigcap_{\fkp \in \Ass_RM,\ \dim R/\fkp \geq d_{i+1}}M(\fkp)$
for all $0\leq i < \ell$. 
\item[$(2)$] Let $1\leq i \leq \ell$. Then 
$\Ass_RC_i=\{ \fkp \in \Ass_RM \mid \dim R/\fkp =d_i\}$ and
$\Ass_RD_i=\{ \fkp \in \Ass_RM \mid \dim R/\fkp \leq d_i\}$. 
\item[$(3)$] $\Ass_RM/D_i=\{\fkp \in \Ass_RM \mid \dim R/\fkp 
\geq d_{i+1} \}$ for all $1\leq i < \ell$. 
\end{enumerate}
\end{pro}

%%%%%%%%%%%%%%%%%%%%%%%%%%%%%%%%%%%%%%%%%%%%%

We now assume that $R$ is a local ring with maximal ideal $\fkm$ and let $M$ be a finitely generated $R$-module with $d = \dim_RM \ge 1$ and $\calD =\{D_i\}_{0 \le i \le \ell}$ the dimension filtration. Let $\x=x_1,x_2, \ldots, x_d$ be a system of parameters of $M$. Then $\x$ is said to be %{\it good} (resp. 
{\it distinguished}, if  
$$  (x_j \mid d_{i} < j \le d) D_i=(0)$$
%$$(x_j \mid d_{i} < j \le d)M \cap D_i=(0) \ \ \text{(resp.}\ \  (x_j \mid d_{i} < j \le d) D_i=(0))$$ 
for all $1\le i\le \ell$, where $d_i=\dim_R D_i$. A parameter ideal $\fkq$ of $M$ is called  %{\it good} (resp. 
{\it distinguished}, if there exists a %good (resp. 
distinguished system $x_1,x_2, \ldots, x_d$ of parameters of $M$ such that $\fkq=(x_1,x_2, \ldots, x_d)$. Therefore, if $M$ is a Cohen-Macaulay $R$-module, every parameter ideal of $M$ is distinguished. 
%Notice that good parameter ideals are distinguished in due course. The converse is also true, if $M$ is a sequentially Cohen-Macaulay $R$-module  (\cite{CC}).
 Distinguished system of parameters exist %(\cite[Lemma 2.4]{CC}) 
and if $x_1, x_2, \ldots, x_d$ is a distinguished system of parameters of $M$, then $x_1^{n_1},x_2^{n_2}, \ldots, x_d^{n_d}$ is also a distinguished system of parameters of $M$ for all integers $n_j \ge 1$. 

%%%%%%%%%%%%%%%%%%%%%%%%%%%%%%%%%%%%%%%%%%%%%

\medskip
\begin{setting}\label{2.6}{\rm
Let $\x=x_1,x_2,\ldots,x_s$ be a system  of elements of $R$ and $\q_j$ denote the ideal generated by $x_1,\ldots,x_j$ for all $j=1,\ldots,s$.   }
\end{setting}

\begin{dfn}\rm
A system $\x$ of elements of $R$ is called {\it Goto sequence} on $M$, if   for all $0\le j\le s-1$ and $0\le i\le\ell$, we have the following
\begin{enumerate}
\item[$(1)$] $\Ass(C_i/\q_j C_i)\subseteq \Assh(C_i/\q_j C_i)\cup \{\fkm\}$,
\item[$(2)$] $x_jD_i=0$ if $d_i<j\le d_{i+1}$,
\item[$(3)$] $\q_{j-1}:x_j=\H^0_\m(M/\q_{j-1}M) \text{ and } x_j\not\in\fkp \text{ for all } \fkp\in\Ass(M/\q_{j-1}M)-\{\fkm\}$.
%\item[$(4)$] $r_{d-j}(M/\q_jM)\le r_{d-j-1}(M/\q_{j+1}M)$.
\end{enumerate}

\end{dfn}

%%%%%%%%%%%%%%%%%%%%%%%%%%%%%%%%%%%%%%%%%%%%%
At first glance, the definition of normal does not seem very intuitive. Once we
enter the world of sequences, however, we will see that Goto sequence has a very
nice  interpretation and properties. We will also see that Goto sequence is useful for many inductive proofs in the next sections.
Before we can give some properties of this sequence, we first need reformulate the notion of $d$-sequences. The sequence $x_1,x_2,\ldots,x_s$ of elements of $R$ is called a
$d$-sequence on $M$ if
$$\q_i M:x_{i+1}x_j=\q_i M:x_j$$
for all $0\le i<j\le s$. The concept of a $d$-sequence is given by Huneke  and it plays an important role in the theory of Blow up
algebra, e.g. Ress algebra.  In the following lemma, we will give some properties of Goto sequences that will
be used in the next sections when we study the index of reducibility and the Cohen-Macaulayness of local rings.

\medskip

\begin{lem}\label{property}
Let $\x=x_1,x_2,\ldots,x_s$ form a Goto sequence on $M$. Then we have
\begin{enumerate}
\item[$(1)$] $\x$ is part of a system of parameters of $M$. 
\item[$(2)$] $\x$ is a $d$-sequence.
\item[$(3)$] If $d=s$ then $\x$ is a distinguished system of parameters of $M$.
%\item[$(4)$] Let $N$ denote the unmixed component of $M/\q_{d-2}M$. If $M/N$ is Cohen-Macaulay, so is also $M/D_{\ell-1}$. 
\end{enumerate}

\end{lem}

\begin{proof}
As an immediate consequence of the definitions we have the first assertion and the thrid assertion. The second assertion is followed from $(vii)$ of \cite[Theorem 1.1]{T}.

\end{proof}

\begin{lem}\label{property1}	
Let $R$ be a homomorphic image of a Cohen–Macaulay
local ring. Assume that system $\x=x_1,x_2,\ldots,x_d$ of parameters form a Goto sequence on $M$. Let $N$ denote the unmixed component of $M/\q_{d-2}M$ and $d\ge 2$. If $M/N$ is Cohen-Macaulay, so is also $M/D_{\ell-1}$.

\end{lem}
\begin{proof}

We may assume that $d \ge 3$. Because the assumption of the corollary
is inherited to the module $M/\q_iM$, it is enough to prove the following
statement.

Let $x \in R$ be a Goto sequence of length one on $M$. Let $N$ denote the unmixed component of $M/x M$ and $d \ge 3$.
If $\H^i_\m(M/N)=0$ for $i\le d-2$ then $\H^i_\m(C_{\ell})=0$ for $i\le d-1$.

For a submodule $N$ of M, we denote $\overline N=(N+xM)/xM$ the submodule of $M/xM$. Since $x$ is a Goto sequence of length one on $M$,  $\Ass(C_\ell/xC_\ell)\subseteq \Assh(C_\ell/xC_\ell)\cup\{m\}$. Therefore $N/\overline D_{\ell-1}$ has a finite length. Since $\overline M/N$ is a Cohen-Macaulay module, $H^i_\fkm(M/D_{\ell-1}+xM)=0$ for all $0<i<d-1$. Therefore, we derive from the exact sequence
$$0 \to M/D_{\ell-1}\overset{.x}\to M/D_{\ell-1} \to  M/D_{\ell-1}+xM \to 0$$
the following exact sequence:
$$0 \to \H^0_\m(M/D_{\ell-1} + xM) \to \H^1_\m(M/D_{\ell-1})
\overset{.x}\to \H^1_\m(M/D_{\ell-1}) \to 0.$$
Thus  $\H^1_\fkm(M/D_{\ell-1})=0$, and so $N/\overline D_{\ell-1}=\H^0_\fkm(M/D_{\ell-1}+xM)=0$. Hence $N=\overline D_{\ell-1}$. Moreover, since $x$ is  $C_\ell=M/D_{\ell-1}$-regular and  $C_\ell/xC_\ell\cong \overline M/\overline D_{\ell-1} = \overline M/N$ a Cohen-Macaulay module, $C_\ell$  is a Cohen-Macaulay module.
\end{proof}

We now denote $r_j(M)=\ell_R((0):_{\H^j_\fkm(M)}\fkm)$ for all $j\in \Bbb Z$ and 
$$r(M)=\sum\limits_{j\in\Bbb Z}r_j(M).$$

\begin{dfn}\rm
A system $\x$ of elements of $R$ is called {\it Goto sequence of type II} on $M$, if we have   
 $$r(M/\q_jM)\le r(M/\q_{j+1}M),$$ 
for all $0\le j\le s-1$.

\end{dfn}

Now, we explore the existence of Goto sequence of type II. We have divided the proof of the existence of Goto sequence of type II into sequence of lemmas. First, we begin with the following result of S. Goto and Y. Nakamura  \cite{GN}.% is often used  in this section.
 
\begin{lem}{\cite{GN}}\label{finitely}
Let $R$ be a homomorphic image of a Cohen-Macaulay local
ring and assume that $\Ass(R)\subseteq \Assh(R)\cup \{\fkm\}$. Then
$$\calF=\{\fkp\in\Spec(R)\mid\height_R(\fkp)> 1=\depth(R_\fkp)\}$$
is a finite set.
\end{lem}

The next proposition shows the existence of a special  element which is useful for the existence of Goto sequence. %and many inductive proofs in the sequel.

\begin{prop}\label{exists element}
Let $R$ be a homomorphic image of a Cohen-Macaulay
local ring and $I$  an $\fkm$-primary ideal of $R$. Assume that $\calF=\{M_i\}_{i=0}^{\ell}$ is a finite filtration of submodules of $M$ such that $\Ass L_i\subseteq \Assh L_i\cup\{\fkm\}$, where $L_i=M_i/M_{i-1}$.  Then there exists an element $x\in I$ 
%such that $(0):_{L_i}x=\H^0_\m(L_i)$, $(0):_Mx=\H^0_\m(M)$ and $\Ass(L_i/x^n L_i)\subseteq \Assh(L_i/x^n L_i)\cup \{\fkm\}$ for all $i=0,\ldots,\ell-1$, and $x\not\in \fkp$ for all $\fkp\in\Ass(M)-\{\fkm\}$.
satisfies the following conditions
%\begin{enumerate}
%\item[$(1)$] For all $i=0,\ldots,\ell-1$,  $\Ass(L_i/x^n L_i)\subseteq \Assh(L_i/x^n L_i)\cup \{\fkm\}$. 
%\item[$(2)$] For all $\fkp\in\Ass(M)-\{\fkm\}$, $x\not\in \fkp$. 
%\item[$(3)$] For all $i=0,\ldots,\ell-1$, $(0):_{L_i}x=\H^0_\m(L_i)$, $(0):_Mx=\H^0_\m(M)$.
%\end{enumerate}
\begin{enumerate}
\item[$(1)$]   $\Ass(L_i/x^n L_i)\subseteq \Assh(L_i/x^n L_i)\cup \{\fkm\}$, For all $i=0,\ldots,\ell-1$.
\item[$(2)$] $x\not\in \fkp$,  for all $\fkp\in\Ass(M)-\{\fkm\}$.
\item[$(3)$]  $(0):_{L_i}x=\H^0_\m(L_i)$ and $(0):_Mx=\H^0_\m(M)$, for all $i=0,\ldots,\ell-1$,
\end{enumerate}

\end{prop}

\begin{proof}

Set $I_i=\Ann(L_i)$, and $R_i=R/I_i$, then  $\Ass(R_i)\subseteq \Assh(R_i)\cup\{\fkm\}$ and $\dim R/I_i> \dim R/I_{i+1}$ for all  $i=0,\ldots,s-1$. Moreover, we have 
$$\Ass(R_i)=\Ass(L_i)=\{\fkp\in\Spec(R)\mid\fkp\in\Ass(M)\text { and } \dim R/\fkp=\dim R/I_i=d_i\}\cup\{\fkm\}.$$
Set  $$\calF_i=\{\fkp\in\Spec(R)\mid I_i\subset\fkp \text{ and } \height_{R_i}(\fkp/I_i)> 1=\depth((L_i)_\fkp)\}.$$
By Lemma \ref{finitely} and the fact $\Ass(L_i)\subseteq\Assh(L_i)\cup\{\m\}$, we see that the set $$\{\fkp\in\Spec(R_i)\mid \height_{R_i}(\fkp)> 1=\depth((L_i)_\fkp)\}$$
is finite, and so that $\calF_i$ are a finite set for all $i=1,\ldots,\ell$. Put $\calF=\Ass(M)\cup\bigcup\limits_{i=1}^t\calF_i\setminus\{\fkm\}$. By the Prime Avoidance Theorem, we can choose $y\in I$ such that   $y\not\in\bigcup\limits_{\fkp\in\calF }\fkp$ and $\dim M_i/y M_i=\dim M_i-1$ for all $i=1,\ldots,\ell$.  On the other hand,  
we can choose an integer $n_0$ such that $(0):_M y^n=(0):_M y^{n_0}$ and $(0):_{L_i} y^n=(0):_{L_i} y^{n_0}$, for all $n\ge n_0$ and $i=1,\ldots,\ell$. Put $x=y^{n_0+1}$. Then we have $x\not\in\bigcup\limits_{\fkp\in\calF }\fkp$ and  $(0):_{L_i} x^2=(0):_{L_i} x$ for all $i=1,\ldots,\ell$. 
Now we show that $x$ have the conditions as required.  

First let us  prove the condition $(1)$. To this end, consider $\fkp\in\Ass(N_i/x N_i)$ with $\fkp\not=\fkm $. Then we have $\depth(L_i/xL_i)_\fkp=0$. On the other hand, $\depth(M_i)_\fkp>0$ since $\fkp\not\in \Ass(L_i)\subseteq \Ass(M)$. Hence $\depth(L_i)_\fkp=1$. It implies that $\height_{R_i}(\fkp)=1$, since $\fkp\not\in\calF_i$. By the assumption $R_i$ is a catenary ring, therefore 
$$\dim R/\fkp= \dim R_i-\height_{R_i}(\fkp)=\dim R_i/x R_i=\dim L_i/x L_i.$$ Hence $\fkp\in \Assh (L_i/x L_i)$.

Since the condition $(2)$ is trivial, it remains to prove the condition $(3)$. Take $\fkp\in\Ass_R(0):_{L_i}x$ with $\fkp\not=\fkm$, Then $x\not\in\fkp$ as $\fkp\in \Assh(L_i)$. Hence $((0):_{L_i}x)_\fkp=(0)$ and this is a contradiction. It implies that $(0):_{N_i}x$ is finite length. Since $(0):_{L_i} x^2=(0):_{L_i} x$, we have $(0):_{L_i} x=\H^0_\fkm(L_i)$. It follows from the following exact sequence
$$0\to M_{i-1}\to M_i\to L_{i}\to 0$$
and $x\H^0_\fkm(L_i)=0$ for all $i=1,\ldots,\ell$ that the following sequence
$$0\to \H^0_\fkm(M_{i-1})\to \H^0_\fkm(M_i)\to \H^0_\fkm(L_{i})$$
$$\text{and  } 0\to (0):_{M_{i-1}}x\to (0):_{M_{i}}x\to (0):_{L_{i}}x $$
are exact. By induction and $(0):_Mx=(0):_Mx^2$, we have $(0):_Mx=\H^0_\fkm(M)$ and this completes the proof.
\end{proof}

The existence of Goto sequence is established by our next Corollary.

\begin{cor}\label{exists001}
Assume that $R$ is a homomorphic image of a Cohen-Macaulay
local ring and $I$  an $\fkm$-primary ideal of $R$.  Then there exists a system $\x=x_1,x_2,\ldots,x_s$ of elements of $I$ such that $\x$ is a Goto sequence on $M$.

\end{cor}

\begin{proof}
We prove this by induction on $s$, the case in which $s=1$ having been dealt  with in Lemma \ref{exists element}. So we suppose that $s=j\ge 2$ and that the result has been proved for smaller values of $s$. Suppose that $d_{i}<j\le d_{i+1}$ for some $i$. We see immediately from this induction hypothesis that
% By induction we have system $x_1,\ldots,x_j$ of $R$ such that
%satisfies the following conditions
$$\Ass(N_i/\q_{j-1} N_i)\subseteq \Assh(N_{j-1}/\q_{j-1}N_i)\cup \{\fkm\},$$
where $\q_{j-1}=(x_1,\ldots,x_{j-1})$, for all $i=0,\ldots,\ell-1$. Moreover the sequence $x_1,x_2,\ldots,x_{d_i}$ is a system of  parameters  of $D_i$. Therefore $\Ann(D_i)+\q_{j-1}$ is $\m$-primary ideals. So that, by Lemma \ref{exists element}, there exists an element $x_j\in I\cap \Ann(D_i)$, as required. This completes the inductive step, and the proof.
\end{proof}

Let $\fkq = (x_1, x_2,\ldots,x_d)$ be a parameter ideal in $R$ and let $M$ be an $R$-module. For
each integer $n\geq 1$ we denote by $\x^n$ the sequence $x^n_1, x^n_2,\ldots,x^n_d$. Let $K^{\bullet}(x^n)$ be the
Koszul complex of $R$ generated by the sequence $\x^n$ and let
$H^{\bullet}(\x^n;M) = H^{\bullet}(\Hom_R(K^{\bullet}(\x^n),M))$
be the Koszul cohomology module of $M$. Then for every $p\in\Bbb Z$ the family $\{H^p(\x^n;M)\}_{n\ge 1}$
naturally forms an inductive system of $R$-modules, whose limit
$$H^p_\fkq=\lim\limits_{n\to\infty} H^p(\x^n;M)$$
is isomorphic to the local cohomology module
$$H^p_\fkm(M)=\lim\limits_{n\to\infty} \Ext_R^p(R/\fkm^n,M)$$
For each $n\geq 1$ and $p \in\Bbb Z$ let $\phi^{p,n}_{\x,M}:H^p(\x^n;M)\to H^p_\fkm(M)$ denote the canonical
homomorphism into the limit.

%%%%%%%%%%%%%%%%%%%%%%%%%%%%%%%%%%%%%%%%%%%%%%%%%%%%%%%%%%%%%%%%%%%%%
%%%%%%%%%%%%%%%%%%%%%%%%%%%%%%%%%%%%%%%%%%%%%%%%%%%%%%%%%%%%%%%%%%%%%

With this notation we have the following result.

\begin{lem}[\cite{GS1}, Lemma 1.7]\label{split}  Let  $M$  be a finitely generated $R$-module and $x$ an $M$-regular element and $\x=x_1,\ldots,x_r$ be a system of elements in $R$ with $x_1= x$. Then there exists a splitting  exact sequence for each $p \in\Bbb Z$,
$$0\to H^p(\x;M)\to H^p(\x;M/xM)\to H^{p+1}(\x;M)\to0.$$ 
\end{lem}

%%%%%%%%%%%%%%%%%%%%%%%%%%%%%%%%%%%%%%%%%%%%%%%%%%%%%%%%%%%%%%%%%%%%%
%%%%%%%%%%%%%%%%%%%%%%%%%%%%%%%%%%%%%%%%%%%%%%%%%%%%%%%%%%%%%%%%%%%%%

\begin{dfn}[\cite{GSa} Lemma 3.12]\label{sur}
{\rm Let $R$ be a Noetherian local ring with the maximal ideal $\fkm$ and $
\dim R=d \ge1$. Let $M$ be a finitely generated $R$-module. Then there exists an integer $n_0$ 
such that for all systems of parameters $\x=x_1,\ldots,x_d$  for $R$ contained in $\fkm^{n_0}$ and for all $p\in \Bbb Z$, the canonical homomorphisms
$$\phi^{p,1}_{\x,M}:H^p(\x,M)\to H^p_\fkm(M)$$
into the inductive limit are surjective on the socles.
The least integer $n_0$ with this property is called a Goto number of $R$-module $M$ and denote by $\rmg(M)$.}
\end{dfn}

%%%%%%%%%%%%%%%%%%%%%%%%%%%%%%%%%%%%%%%%%%%%%%%%%%%%%%%%%%%%%%%%%%%%%
%%%%%%%%%%%%%%%%%%%%%%%%%%%%%%%%%%%%%%%%%%%%%%%%%%%%%%%%%%%%%%%%%%%%%

We need the following result in next section.

\begin{lem}\label{2.10}
Assume that $N$ is a submodule of $M$ such that $M/N$ is Cohen-Macaulay and $\dim N<\dim M$. Then  for all parameter ideals $\q\subseteq\m^{\rmg(M)}$ of $M$, we have 
$$\calN(\q;M)=\calN(\q;N)+\calN(\q;M/N).$$
%$$ \text{ and }[\q M+N]:_M \fkm=[\q M:_M\fkm]+N,$$
%for all $n\ge 0$. 
\end{lem}
\begin{proof}
Let $\q=(x_1,x_2\ldots, x_d)$ be a parameter ideal of $M$ such that $\q\subseteq \m^n$. Then by the definition of Goto number,  the canonical map
$$\xymatrix{\phi_{M}:M/(x_1,x_2, \ldots, x_{d})M\ar[r]& \H^{d}_\m(M)=\lim\limits_{q\to\infty}M/(x_1^q,x_2^q, \ldots,x_{d}^q)M}$$
is surjective on the socles.  We put $\calM= M/N$ and  look at the exact sequence
$$\xymatrix{0\ar[r]& N\ar[r]^\iota& M \ar[r]^\epsilon & \calM \ar[r]&0}$$
of $R$-modules, where $\iota$ (resp. $\epsilon$) denotes the embedding (resp. the canonical epimorphism).
 Then, since $\dim M> \dim_R N$ and since $x_1,x_2, \ldots, x_d$ is a regular sequence for $\calM$, we get the following commutative diagram 
$$\xymatrix{&&& 0\ar[d]&\\
0\ar[r]& N/\q N\ar[r]^{\overline \iota}& M/\q M \ar[d]^{\phi_M}\ar[r]^{\overline \epsilon} & \calM/\q \calM \ar[r] \ar[d]^{\phi_{\calM}}&0\\
& & \H^d_\m(M) \ar[r]^{=} & \H^d_\m(\calM) &}$$
with  exact first row. Let $x\in (0):_{\calM/\q \calM}\m$. Then, since $\phi_M$ is surjective on the socles, we get  an element $y\in (0):_{M/\q M}\m$ such that $\phi_{\calM}(x)=\phi_M(y)$. Thus $\overline{\epsilon}(y)=x$, because the canonical map $\phi_{\calM}$ is injective, whence
$$[N+\q M]:_M \fkm=N + [\q M:_M\fkm].$$
Moreover we have
$$\calN(\q;M)=\calN(\q;N)+\calN(\q;M/N).$$
\end{proof}

%%%%%%%%%%%%%%%%%%%%%%%%%%%%%%%%%%%%%%%%%%%%%%%%%%%%%%%%%%%%%%%%%%%%%
%%%%%%%%%%%%%%%%%%%%%%%%%%%%%%%%%%%%%%%%%%%%%%%%%%%%%%%%%%%%%%%%%%%%%

\begin{lem}\label{2.7}
 Let $M$ be a finitely generated $R$-module. Assume that $x$ is an $M$-regular element of $M$ such that $x\in \m^{\rmg(M)}$. Then
we have %$$\rmg(M/xM)\le \rmg(M),$$and
 $$r_i(M)\le r_{i-1}(M/xM)$$
for all $i\in \Bbb Z$
\end{lem}

\begin{proof}

Let $x_2,\ldots,x_d$ be a system of parameters of module $M/xM$ such that $x_i\in\m^{\rmg(M)}$. Put $\x=x_1,x_2,\ldots,x_d$ and $\q=(\x)$, where $x_1=x$. Since $x\in\m^{\rmg(M)}$, we have $\q\subseteq \m^{\rmg(M)}$.
By the definition of Goto number, we have  the canonical homomorphism
$$H^i(\x,M)\to H^i_\fkm(M)$$
into the inductive limit are surjective on the socles, for each $i\in \Bbb Z$. By the regularity of $x=x_1$ on $M$, it follows from the following sequence 
$$\xymatrix{0\ar[r]&M\ar[r]^{.x}&M\ar[r]&M/xM\ar[r]&0}$$
that there are induced the diagram
$$\xymatrix{0\ar[r]&H^i(\x;M)\ar[d]\ar[r]&H^i(\x,M/xM)\ar[r]\ar[d]&H^{i+1}(\x;M)\ar[r]\ar[d]&0\\
\ar[r]&H^i_\fkm(M)\ar[r]&H^i_\fkm(M/xM)\ar[r]&H^{i+1}_\fkm(M)\ar[r]&}$$
commutes, for all $i\in \Bbb Z$. 
 It follows from the above commutative diagrams and Lemma \ref{split} that after applying the functor $\Hom(k,*)$, we obtain the commutative diagram 
$$\xymatrix{\Hom(k,H^i(\x,M/xM))\ar[r]\ar[d]&\Hom(k,H^{i+1}(\x;M))\ar[r]\ar[d]&0\\
\Hom(k,H^i_\fkm(M/xM))\ar[r]&\Hom(k,H^{i+1}_\fkm(M))}$$
for all $i\in \Bbb Z$. Since the map $\Hom(k,H^{i+1}(\x;M))\to\Hom(k,H^{i+1}_\fkm(M))$ is surjective, so is the map $\Hom(k,H^i_\fkm(M/xM))\to\Hom(k,H^{i+1}_\fkm(M))$. Therefore the map $\Hom(k,H^{i}(\x;M/xM))\to\Hom(k,H^{i}_\fkm(M/xM))$ is surjective and $r_i(M)\le r_{i-1}(M/xM)$
for all $i\in \Bbb Z$. This completes the proof.

 %Thus for all systems $\x$ of parameters of module $M/xM$ such that $x_i\in\m^{\rmg(M)}$, we have the map $\Hom(k,H^{i}(\x;M/xM))\to\Hom(k,H^{i}_\fkm(M/xM))$ is surjective for all $i\in \Bbb Z$. Hence we have   $$r_i(M)\le r_{i-1}(M/xM)$$ for all $i\in \Bbb Z$, as required.

\end{proof}

\begin{lemma}\label{23}
Let $M$ be a finitely generated $R$-module with $\dim M=s\ge 1$.
Let $k$ and $l$ be two positive integers.  Then there
exists an integer $n_3> l$ such that
$$(\frak m^{n_3}+\H^0_\m(M)):\frak m^{k}\subseteq\frak m^{l}M+\H^0_\m(M).$$

\end{lemma}
\begin{proof}
Let $\overline M=M/\H^0_\m(M)$. Then there is an $\overline{M}$-regular element $a$ contained in $\frak m^{k}$. By the Artin-Rees Lemma, there
exists a positive integer $m$ such  that $\frak m^{l+m}\overline{M}
\cap a\overline{M} =\frak m^l(\frak m^{m}\overline{M}\cap a\overline{M} )$. Set $n_3=l+m$. We have
$$a(\frak m^{n_3}\overline{M}:\frak m^k)\subseteq a(\frak m^{n_3}\overline{M} :a)=\frak m^{n_3}\overline{M} \cap a\overline{M} =\frak m^{l}(\frak m^m\overline{M} \cap a\overline{M} ),$$
and so  $a(\frak m^{n_3}\overline{M} :\frak m^k)\subseteq a\frak m^{l}
\overline M$. It follows from the regularity of  $a$ that
$\frak m^{n_3}\overline{M} :\frak m^k\subseteq \frak m^{l} \overline{M} $.
Hence $(\frak m^{n_3}M+\H^0_\m(M)):{\frak m}^k\subseteq \frak m^{l}
M+\H^0_\m(M)$, as required.
\end{proof}

\begin{lemma}\label{251a}Let $M$ be a finitely generated $R$-module with $\dim M=s\ge 1$.
Then there exists a positive integer $n_4$ such that  for all $x\in \frak m^{n_4}$, we have
 $$r_0(M)+r_1(M)\le r_0(M/xM).$$
\end{lemma}
\begin{proof}
Since $H^0_\frak m(M)$ have finite length, there exists an integer $l$ such that
$\frak m^l M\cap \H^0_\frak m(M)=0$. By Lemma \ref{23}, there is an integer $n_4> \max\{l,\rmg(M/\H^0_\m(M))\}$ such
that for all $x\in\frak m^{n_4}$  we have
$$(x M+\H^0_\frak m(M)):\frak m\subseteq (\frak m^{n_4}M+\H^0_\frak m(M)):\frak m\subseteq \frak m^{l}M+\H^0_\frak m(M).$$ 
 Let $b\in(x M+\H^0_\frak m(M)):\frak m$. Write $b=\alpha+\beta$ with $\alpha\in \frak m^{l}M$ and $\beta\in \H^0_\frak m(M)$.
 Then,  since $x\in \frak m^{n_4}\subseteq\frak m^{l+1}$, we get that
  $$\frak m\alpha\subseteq \frak m^{l+1}M\cap(xM+\H^0_\frak m(M))=x M+\frak m^{l+1}M\cap \H^0_\frak m(M)=xM.$$
  Thus $\alpha\in x M:\frak m$, and so  $(x M+\H^0_\frak m(M)):\frak m
 = x M:\frak m+\H^0_\frak m(M)$.
 Since $\frak m^l M\cap H^0_\frak m(M)=0$, we have the following exact sequence
 $$0\to \H^0_\frak m(M)\to M/x M\to M/(x M+\H^0_\frak m(M))\to 0.$$
It follow that the sequence
$$0\to (0):_M\frak m\to (x M:\frak m)/x M\to ((x M+\H^0_\frak m(M)):\frak m)/(x M+\H^0_\frak m(M))\to 0$$
is exact. Therefore, for all   $x\in \frak m^{n_4}$ we have  
$$r_0(M)+r_0(M/xM+\H^0_\m(M))\le r_0(M/xM).$$
Since $n_4\ge \rmg(M/\H^0_\m(M))$, by Lemma \ref{2.7}, we have $r_1(M/H^0_\m(M))\le r_0(M/xM+\H^0_\m(M))$.
Since $r_1(M)=r_1(M/H^0_\m(M))$, we have
$$r_0(M)+r_1(M)\le r_0(M/xM),$$
as required.

\end{proof}

\begin{cor}\label{leCMt}
 Let $M$ be a finitely generated $R$-module with $\dim M\ge 2$. Then there exists an integer $n$ such that for all parameter elements $x\in \m^n$, we have
$$r(M)\le r(M/xM).$$
\end{cor}
\begin{proof}
Choose $n_1$, as lager as possible, such that $n_1\ge \rmg(M/\H^0_\m(M))$ and for all parameter $x\in \frak m^{n_1}$, we have
$$r_0(M)+r_1(M)\le r_0(M/xM).$$
Since $\H^i_\m(M)=\H^i_\m(M/\H^0_\m(M))$ for all $i\ge 1$, by Lemma \ref{2.7}, we have $r_i(M)\le r_{i-1}(M/xM)$ for all $i\ge 2$. Hence for all parameter $x\in \frak m^{n_1}$, we have
$$r(M)\le r(M/xM),$$
as required.
\end{proof}

%\begin{cor}\label{leCMt0}
% Let $M$ be a finitely generated $R$-module with $\dim M\ge 2$. Then there exists an integer $n$ such that for all parameter elements $x\in \m^n$, we have
%$$r_d(M)\le r_{d-1}(M/xM).$$
%\end{cor}
%\begin{proof}
%Since $\dim M\ge 2$ and $x$ is a parameter element of $M$, we have $\H^d_\m(M)= \H^d_\m(M/\H^0_\m(M))$ and $\H^{d-1}_\m(M/xM)=\H^{d-1}_\m(M/xM+\H^0_\m(M))$. Therefore we have been working under the assumption that $\depth M\ge 0$. Then by Lemma \ref{2.7}, we have $$r_d(M)\le r_{d-1}(M/xM),$$
%and the proof is complete.

%\end{proof}

%%%%%%%%%%%%%%%%%%%%%%%%%%%%%%%%%%%%%%%%%%%%%%%%%%%%%%%%%%%%%%%%%%%%%
%%%%%%%%%%%%%%%%%%%%%%%%%%%%%%%%%%%%%%%%%%%%%%%%%%%%%%%%%%%%%%%%%%%%%

The existence of Goto sequence of type II is established by our next Proposition.

\begin{prop}\label{exists}
Assume that $R$ is a homomorphic image of a Cohen-Macaulay
local ring and $I$  an $\fkm$-primary ideal of $R$.  Then there exists a system $\x=x_1,x_2,\ldots,x_s$ of elements of $I$ such that $\x$ is a Goto sequence of type II on $M$.

\end{prop}
\begin{proof}
We shall now show the our result by induction on $s$. In the case in which $s=1$ there is nothing to prove, because of the Lemma \ref{exists element}. So assume inductively that $j \in\Bbb N$ with $j > 1$ and and that the desired
result has been established when $s = j -1$. Suppose that $d_{i}<j\le d_{i+1}$ for some $i$. By induction we have system $x_1,\ldots,x_{j-1}$ of $R$ such that satisfies the following conditions
\begin{enumerate}
\item[$(1)$] $\Ass(N_i/\q_{j-1} N_i)\subseteq \Assh(N_j/\q_{j-1}N_i)\cup \{\fkm\}$, where $\q_{j-1}=(x_1,\ldots,x_{j-1})$, for all $i=0,\ldots,\ell-1$. 
\item[$(2)$] The sequence $x_1,x_2,\ldots,x_{d_i}$ is a system of  parameters  of $D_i$.
%\item[$(3)$] $\q_{j-1}:x_j=\H^0_\m(M/\q_{j-1}M) \text{ and } x_j\in\fkp \text{ for all } \fkp\in\Ass(M/\q_{j-1}M)-\{\fkm\}.$
\end{enumerate}
Let $\overline R=R/\q_{j-1}$ $\overline M=M/\q_{j-1}M$ $\fkn=\m/\fkq_{j-1}$.  %Assume that $\calD^\prime = \{D^\prime_i\}_{0 \le i \le \ell^\prime}$ is the dimension filtration of $\overline M$ with $d^\prime_i=\dim D^\prime_i$. 

Choose $n$, as lager as possible, such that for all parameter $x\in \frak m^{n}$, we have
$$r(\overline M)\le r(\overline M/x \overline M).$$
%Since $\H^i_\m(M)=\H^i_\m(M/\H^0_\m(M))$ for all $i\ge 1$, by Lemma \ref{2.7}, we have $r_i(M)\le r_{i-1}(M/xM)$ for all $i\ge 2$. Hence for all parameter $x\in \frak m^{n_1}$, we have
%$$r(M)\le r(M/xM),$$
Put $J=(\Ann(D_i)+\q_{j-1})\cap I\cap \m^n$. Then $J\overline R$ is an $\fkn$-primary ideal of $\overline R$. 
 By Lemma \ref{exists element}, we can choose $x_{j+1}\in\Ann(D_i)\cap I\cap\m^n$, as required. With this observation, we can complete the inductive step and the proof.
\end{proof}

%%%%%%%%%%%%%%%%%%%%%%%%%%%%%%%%%%%%%%%%%%%%%%%%%%%%%%%%%
%%%%%%%%%%%%%%%%%%%%%%%%%%%%%%%%%%%%%%%%%%%%%%%%%%%%%%%%%
%%%%%%%%%%%%%%%%%%%%%%%%%%%%%%%%%%%%%%%%%%%%%%%%%%%%%%%%%
%%%%%%%%%%%%%%%%%%%%%%%%%%%%%%%%%%%%%%%%%%%%%%%%%%%%%%%%%
%%%%%%%%%%%%%%%%%%%%%%%%%%%%%%%%%%%%%%%%%%%%%%%%%%%%%%%%%

\section{The Proof of Main theorem}

The purpose of this section is to prove the equivalence of assertions (1) and (2) in  Theorem \ref{0}. We maintain the following settings.

\begin{setting}\label{2.6}{\rm
Assume that $R$ is a homomorphic image of a Cohen-Macaulay local ring. Let $\calD = \{D_i\}_{0 \le i \le \ell}$ be the dimension filtration of $M$ with $\dim D_i=d_i$. We put  $L = M/D_{\ell -1}$. 
}
\end{setting}

The notion of a sequentially Cohen-Macaulay module was introduced first by Stanley \cite{St} for the graded case and in \cite{Sch} for the local case.

\begin{dfn}[\cite{Sch, St}]\rm
We say that $M$ is a {\it sequentially Cohen-Macaulay $R$-module}, 
if $C_i$ is a Cohen-Macaulay $R$-module for all 
$1\leq i \leq \ell$, where $C_i=D_i/D_{i-1}$. 
\end{dfn}

Recall that an $R$-submodule $N$ of $M$ is {\it irreducible}, if $N$ is not written as the intersection of two larger $R$-submodules of $M$. Every $R$-submodule $N$ of $M$ can be expressed as an irredundant intersection of irreducible $R$-submodules of $M$ and the number of irreducible $R$-submodules appearing in such an expression depends only on $N$ and not on the expression. Let us call, for each parameter ideal $\q$ of $M$, the number $\calN(\q;M)$ of irreducible $R$-submodules of $M$ that appear in an irredundant irreducible decomposition of $\q M$ the index of reducibility of $M$ with respect to $\q$. Let $S$ be the set of parameter ideals of $M$ and $a$ integer number. Then we say that the index of reducibility for $S$ are {\it eventually $a$}  if there some integer number $n$ such that for all parameter ideals $\q\in S\cap \m^n$, we have $$\calN(\q;M)=a.$$  

We then have  the following.

\begin{theorem}[{\cite[Theorem 1.1]{Tr}}]\label{3.800}
Suppose that $M$ is a sequentially Cohen-Macaulay $R$-module. Then  the index of reducibility for all distinguished parameter  ideals of $M$ are eventually $r(M)$.
\end{theorem}

Our next major aim is the establishment of a proof of converse of the above Theorem. To prepare the ground for this, we begin the following result.

\begin{proposition}\label{P4.60}
 Assume that $d\ge 2$ and the index of reducibility for all distinguished parameter  ideals of $M$ are eventually $r(M)$. 
Then $L$ is Cohen-Macaulay.
\end{proposition}
\begin{proof}
Since the index of reducibility for all distinguished parameter  ideals of $M$ are eventually $r(M)$, there exists  an integer $n$ such that for all distinguish parameter ideals $\fkq\subseteq \frak m^{n}$ we
    have
$$\calN(\q;M)=r(M).$$
By Proposition \ref{exists}, there exists a Goto sequence $x_1,\ldots,x_{d-2}$ of type II in  $\m^n$.
%  such that the following conditions hold true.
%\begin{enumerate}
%\item[$(1)$] $\Ass(N_i/\q_j N_i)\subseteq \Assh(N_i/\q_j N_i)\cup \{\fkm\}$, where $\q_j=(x_1,\ldots,x_j)$, for all $i=0,\ldots,\ell-1$ and $j=1,\ldots,d-2$. 
%\item[$(2)$] $x_jD_i=0$ for all $d_i<j\le d_{i+1}\le d-2$ and $i=0,\ldots,\ell-1$. 
%\item[$(3)$] $\q_{j-1}:x_j=\H^0_\m(R/\q_{j-1}) \text{ and } x_j\in\fkp \text{ for all } \fkp\in\Ass(R/\q_{j-1}R)-\{\fkm\},$  for all $j=1,\ldots,d-2$.
%\item[$(4)$] $r_{d-j}(R/\q_j)\le r_{d-j-1}(R/\q_{j+1})$ for all $j=0,\ldots,d-3$.
%\end{enumerate}
Let  $\q_{d-2}=(x_1,\ldots,x_{d-2})$ and $A=M/\q_{d-2}M$ and let $N$ denote the unmixed component of $A$. Then $A/N$ is a generalized Cohen-Macaulay $R$-module since $\dim A/N=2$ and $A/N$ is unmixed. Therefore there exists an integer $n_0>n$  such that for all parameters  $x\in\m^{n_0}$, we have $x H^1_\m(A/N)=0$, $r_1(A/(x)+N)=r_1(A/N)+r_2(A/N)$ and $\calN(x;N)=r_0(N)+r_1(N)$, because $\dim N\le1$.  Suppose that $d_{i_0}<d-2\le d_{i_0+1}$ for some $i_0$. 
Then $\Ann(\fka_{i_0})+\q_{d-2}$ is an $\m$-primary ideal of $R$. Then we can choose $x_{d-1}\in \m^{n_0}\cap \Ann(\fka_{i_0})$ as in Proposition \ref{exists element}.

Let $\q_{d-1}=(\q_{d-2},x_{d-1})$ and $B=M/\q_{d-1}M$. Since $\dim B=1$, $B$ is sequentially Cohen-Macaulay. It follows from $\Ann \fka_{\ell-1}+\q_{d-1}$ is an $\m$-primary ideal of $R$ and Theorem \ref{3.800}, we have choose $x_d\in\Ann \fka_{\ell-1}$ such that 
$$\calN(x_dB;B)=r_1(B)+r_0(B).$$

On the other hand, it follows from the exact sequence
$$0\to N\to A\to A/N\to 0$$
and $x_{d-1}$ is a regular on $A/N$ that $r_2(A/N)=r_2(A)$ and
$$0\to N/x_{d-1}N\to B\to A/(x_{d-1})+N\to 0.$$
Since $\dim N/x_{d-1}N=0$, the sequence
$$0\to \H^0_\m(N/x_{d-1}N)\to \H^0_\m(B)\to \H^0_\m(A/(x_{d-1})+N)\to 0.$$
is exact and  $\H^1_\m(B)=\H^1_\m(A/(x_{d-1})+N)$. Thus  $$r_1(B)=r_1(A/(x_{d-1})+N)=r_1(A/N)+r_2(A/N)=r_1(A/N)+r_2(A),$$
because of the choice of $x_{d-1}$.
It follows from the above exact sequence that the sequence
$$0\to (0):_{N/x_{d-1}N}\m \to (0):_{\H^0_\m(B)}\m\to (0):_{\H^0_\m(A/x_{d-1}A+N)}\m.$$
is exact and so $r_0(N)+r_1(N)=\calN(x_{d-1};N)\le r_0(B)$, because of the choice of $x_{d-1}$.
Since $r_0(B)+r_1(B)=\calN(x_d;B)$, therefore  we have 
$$r_0(N)+r_1(N)+r_1(A/N)+r_2(A)\le \calN(x_d;B)= \calN(\q;M)$$
By the definition of Goto sequence of type II, we have $\calN(\q;M)\le r(M)\le r(A)$. Thus  $r_1(N)+r_1(A/N)\le r_1(A)$ because $r_0(N)=r_0(A)$ and $r(A)=r_0(A)+r_1(A)+r_2(A)$. On the other hand, it follows from the exact sequence
$$0\to N\to A\to A/N\to 0$$
that the sequence
$$0\to \H^1_\m(N)\to \H^1_\m(A)\to \H^1_\m(A/N)\to 0$$
is exact. Thus $r_1(A)\le r_1(N)+r_1(A/N)$ and so $ r_1(N)+r_1(A/N)=r_1(A)$.
 It follows that
$$r_0(N)+r_1(N)+r_1(A/N)+r_2(A)= \calN(x_d;B).$$
Since $r_0(B)+r_1(B)=\calN(x_d;B)$ and $r_1(B)=r_1(A/N)+r_2(A)$, therefore  we have 
$\calN(x_{d-1};N)=r_0(B)$. It follows from the above exact sequence that $\H^0_\m(A/x_{d-1}A+N)=0$. 
Now we derive from the exact sequence
$$0 \to A/N\overset{.x_{d-1}}\to A/N \to  A/N+x_{d-1}A \to 0$$
the following exact sequence
$$0 \to \H^0_\m(A/N+x_{d-1}A) \to \H^1_\m(A/N)
\overset{.x_{d-1}}\to \H^1_\m(A/N) .$$
Thus  $\H^1_\fkm(A/N)=0$, because $x_{d-1}\H^1_\m(M)=0$. Hence $L$ is Cohen-Macaulay, because of Lemma \ref{property1}, and the proof is complete.
%\end{proof}
\end{proof}

The next result is of special significance as it is used as a basis for proving the Theorem \ref{0}, and is formulated as follows.

%The following statement provide information that will be helpful later on.

\begin{proposition}\label{P4.61}
Assume that there exists  an integer $n$ such that for all distinguish parameter ideals $\fkq\subseteq \frak m^{n}$ we
    have
$$ \calN(\q;M)\le r(M).$$
Then $R$ is sequentially Cohen-Macaulay.
\end{proposition}
\begin{proof}
 We use induction on the dimensional  $d$ of $M$. In the case in which $\dim M=1$, it is clear that $R$ is sequentially Cohen-Macaulay. Now suppose, inductively, that $d > 1$ and that the result
has been proved for smaller values of $d$.  Recall that $D_{\ell-1}$ is the unmixed component of $M$. Therefore, by the Prime Avoidance Theorem, we can choose the part of a system $x_{d_{\ell-1}+1},\ldots,x_d$ of parameters of $R$ such that $\fkb\subseteq \fkm^n$ and $\fkb M \cap D_{\ell-1}=0$, where $\fkb=(x_{d_{\ell-1}+1},\ldots,x_d)$. 
%Consequently, $(D_i+\fkb M)/\fkb M=D_i$ for all $i=0,\ldots,\ell-1$, and so $\Lambda(M)-\{d\}\subseteq \Lambda(M/\fkb M)$.
 On the other hand, since $\calN(\q;M)\le r(M)$ for all distinguished parameter ideals $\q\subseteq\m^n$, by Proposition \ref{P4.60}, $L$ is Cohen-Macaulay.
A simple inductive argument therefore shows that it is enough for us to prove that there exists $n_1$ such that for all distinguished parameter ideals $\q\subseteq \m^{n_1}$ of $D_{\ell-1}$, we have $$\calN(\q;D_{\ell-1})=r(D_{\ell-1}).$$

In fact, since $L$ is Cohen-Macaulay, straightforwardly by Lemma \ref{2.10}, there exists $n_0$ such that for all distinguished parameter ideals $\q\subseteq \m^{n_0}$, we have
$$\calN(\q;M)=\calN(\q;D_{\ell-1})+\calN(\q;L).$$

Choose $n_1\ge \max\{n_0,n\}$.  Assume that $x_1,\ldots,x_{d_{\ell-1}}$ is a distinguished system of parameters of $D_{\ell-1}$ such that $(x_1,\ldots,x_{d_{\ell-1}})\subseteq \m^{n_1}$. It would then follow from the definition of distinguished system of parameter that system $x_1,\ldots,x_d$ of parameters is distinguished  of $M$ such that $\q\subseteq\m^{n_1}$, where $\q=(x_1,\ldots,x_{d_{\ell-1}},\fkb)$.  Since $n_1\ge n$, we have
$\calN(\q;M)=r(M)=\sum\limits_{j\in\Bbb Z}r_j(M)$. On other hand, it follows from $L$ is Cohen-Macaulay and the  exact sequence
$$0\to D_{\ell-1}\to M \to L\to 0$$
that the following sequence
$$0\to \H^{d}_\m(D_{\ell-1})\to \H^{d}_\m(M) \to \H^{d}_\m(L)\to 0$$
is exact. Since $\dim D_{\ell-1}<d$, we have $\H^{d}_\m(M) \cong \H^{d}_\m(L)$ and so $r_d(M)=r_d(L)$. Since $L$ is Cohen-Macaulay and $\q$ is a parameter ideal of $L$, we have $\calN(\q;L)=r_d(L).$
Therefore, it follows that we obtain $\calN(\q;D_{\ell-1})=r(D_{\ell-1})$. Hence for all distinguished parameter ideals $\q\subseteq \m^n$ of $D_{\ell-1}$, we have $$\calN(\q;D_{\ell-1})=r(D_{\ell-1}).$$

\end{proof}

\begin{proof}[Proof of Theorem $\ref{0}$]
(1) $\Rightarrow$ (2) This is now immediate from Theorem \ref{3.800}.\\
 (2) $\Rightarrow$ (3) is trivial.\\
(3) $\Rightarrow$ (1) This now immediate from  Proposition \ref{P4.61}.

\end{proof}

\begin{cor}\label{C3.6}
%Assume that $R$ is a homomorphic image of a Cohen-Macaulay local ring. Then t
The following statements are equivalent.
\begin{enumerate} 
\item[$({\rm 1})$] $M$ is Cohen-Macaulay.

\item[$({\rm 2})$]  The index of reducibility for all parameter  ideals of $M$ are eventually $r_d(M)$. 

\item[$({\rm 3})$]  The index of reducibility for all distinguished parameter  ideals of $M$ are eventually $r_d(M)$.

\item[$({\rm 4})$]  There exists an integer $n$ such that for all distinguished parameter ideals $\q\subseteq \m^n$, we have
$$\calN(\q;M)\le r_d(M).$$

\end{enumerate}

\end{cor}
\begin{proof}
(1) $\Rightarrow$ (2) This is now immediate from Theorem \ref{0}.\\
(2) $\Rightarrow$ (3) and (3) $\Rightarrow$ (4) are trivial.\\
(4) $\Rightarrow$ (1) Since $r_d(M)\le \sum\limits_{j\in\Bbb Z}r_j(M)=r(M)$, 
 for all distinguished parameter ideals $\q\subseteq \m^n$, we have $\calN(\q;M)\le r(M)$. Thus by Theorem \ref{0}, 
$M$ is sequentially Cohen-Macaulay. By Theorem \ref{3.800}, there exists a distinguished parameter ideal $\q\subseteq$ such that we have 
$r(M)=\calN(\q;M).$ It follows from the hypothesis that $r(M)=r_d(M)$, and so $r_i(M)=0$ for all $i<d$. Hence $M$ is Cohen-Macaulay.

\end{proof}

\begin{theorem}\label{C3.7}
 $R$ is Gorenstein if and only if the index of reducibility for all distinguished parameter  ideals are eventually $1$.

\end{theorem}
\begin{proof}
  $\text {Only if:} $ Since $R$ is Gorenstein, $R$ is Cohen-Macaulay ring and $r_d(R)=1$. By Corollary \ref{C3.6}, the index of reducibility for all distinguished parameter  ideals of $R$ are eventually $1$.

$\text{If}$. Since the index of reducibility for all distinguished parameter  ideals of $R$ are eventually $1$
there exists an integer $n$ such that for all distinguished parameter ideals $\q\subseteq \m^n$, we have
$$\calN(\q;R)=1\le r_d(R).$$
By the Corollary \ref{C3.6}, $R$ is Cohen-Macaulay and so  $r_d(R)=1$.
 Hence $R$ is Gorenstein, as required.

\end{proof}

%\begin{cor}\label{Cmain}
%For all  integers $n$  there exists  a distinguished parameter ideal $\frak q\subseteq \frak m^n$, we have
%$$ r(M)\leqslant \calN(\q;M)$$.
%\end{cor}
%\begin{proof}
%The result follows  from Theorem \ref{0}.
%\end{proof}

%%%%%%%%%%%%%%%%%%%%%%%%%%%%%%%%%%%%%%%%%%%%%%%%%%%%
%%%%%%%%%%%%%%%%%%% References %%%%%%%%%%%%%%%%%%%%


\begin{thebibliography}{GSa3}


%\bibitem [CHV]{CHV} A. Corso, C. Huneke, and W. V. Vasconcelos, {\it On the integral closure of ideals}, Manuscripta Mathematica, No. {\bf 95} (1998), 331-347.
%\bibitem [CP]{CP} A. Corso and C. Polini, {\it Links of prime ideals and their Rees algebras}, J. Algebra {\bf 178} (1995), 224-238.
%\bibitem [CPV] {CPV} A. Corso, C. Polini, and W. V. Vasconcelos, {\it Links of prime ideals}, Math. Proc. Camb. Phil. Soc. {\bf 115} (1994), 431-436.

\bibitem[CC]{CC} D. T. Cuong and N. T. Cuong, \textit{On sequentially Cohen-Macaulay modules}, Kodai Math. J.,  \textbf{30} (2007), 409-428.

%\bibitem[Cu]{Cu} N. T. Cuong, {\it p-standard system of parameters and p-standard ideals in local rings}, Acta Mathematica Vietnamica {\bf 20} (1995), 145-161.


%\bibitem[CGT]{CGT} N. T. Cuong, S. Goto and H. L. Truong, {\it Hilbert coefficients and sequentially Cohen–Macaulay modules}, Journal of Pure and Applied Algebra {\bf 217} (2013) 470-480.

\bibitem[CQT]{CQT} N. T. Cuong, P. H. Quy and H. L. Truong, {\it On the index of reducibility of powers of an ideal},
Accepted for printing in: Journal of Pure and Applied Algebra.

%\bibitem[CST]{CST} N. T. Cuong,  P. Schenzel and  N. V. Trung, {\it Verallgemeinerte Cohen-Macaulay-Moduln}, M. Nachr., {\bf 85} (1978), 57-73.


\bibitem[CT]{CT} N. T. Cuong and H. L. Truong, {\it Asymptotic behavior of parameter ideals in generalized Cohen-Macaulay modules}, J. Algebra, {\bf 320} (2008), 158-168.

%\bibitem[EN] {EN} S. Endo and M. Narita, {\it The number of irreducible components of an ideal and the semi-regularity of a local ring}, Proc. Japan Acad., {\bf 40} (1964), 627-630.

%\bibitem[GGHOPV]{GGHOPV} L. Ghezzi, S. Goto, J.-Y. Hong. K. Ozeki, T. T. Phuong, and W. V. Vasconcelos,{\it  The first Hilbert coefficients of parameter ideals},  J. London Math. Soc. (2) {\bf 81} (2010),679-695.

%\bibitem[G]{G} S. Goto, {\it Hilbert coefficients of parameters}, Proc. of the 5-th Japan-Vietnam Joint Seminar on Commutative Algebra, Hanoi (2010), 1-34.


\bibitem[GN]{GN} S. Goto and Y. Nakamura, {\it Multiplicity and tight closures of parameters}, J. Algebra,  {\bf 244} (2001),  no. 1, 302-311.

%\bibitem[GNi]{GNi} S. Goto and K. Nishida, {\it Hilbert coefficients and Buchsbaumness of associated graded rings}, Journal of Pure and Applied Algebra {\bf 181} (2003) 61-74.

\bibitem[GSa]{GSa}S.  Goto and  H. Sakurai, {\it The equality $I^2 = QI$ in Buchsbaum rings}, Rend. Sem. Mat. Univ. Padova {\bf 110} (2003), 25-56.

\bibitem[GS1]{GS1} S. Goto and  N. Suzuki, {\it Index of Reducibility of Parameter Ideals in a Local Ring,} J. Algebra, {\bf 87} (1984), 53-88.

%\bibitem[GSa1]{GSa1}S. Goto and H. Sakurai, {\it The reduction exponent of socle ideals associated to parameter ideals in a Buchsbaum local ring of multiplicity two}, J. Math. Soc. Japan {\bf 56} (2004) 1157-1168.
%\bibitem[GS1]{GS1} S. Goto and  N. Suzuki, {\it Index of Reducibility of Parameter Ideals in a Local Ring,} J. Algebra, {\bf 87} (1984), 53-88.

%\bibitem[H]{H} J. Horiuchi, {\it Stability of quasi-socle ideals and the structure of their associated graded rings}, Ph D thesis, 2011.
\bibitem[Hu]{Hu} C. Huneke, {\it On the symmetricand Rees algebra of an ideal generated by a d-sequence}, J. Algebra, {\bf 62} (1980), pp. 268-275.

%\bibitem[Kw]{Kw} T. Kawasaki, {\it On Cohen-Macaulayfication of certain quasi-projective schemes,} J. Math. Soc. Japan {\bf 50} (1998), 969-991.

%\bibitem[K]{K} V. Kodiyalam, Homological invariants of powers of an ideal, {\it Proc. Amer. Math. Soc.}, {\bf 118} (1993), 757-764.
%\bibitem[MRS]{MRS}  T. Marley, M.W. Rogers, H. Sakurai, Gorenstein rings and irreducible parameter ideals, {\it Proc. Amer. Math. Soc.} {\bf 136} (2008), 49--53.

%\bibitem[MSV]{MSV} M. Mandal, B. Singh and J. K. Verma, {\it  On some conjectures about the Chern numbers of filtrations}, J. Algebra {\bf 325} (2011), 147-162. 
%\bibitem[MV]{MV} C. Miyazaki and W. Vogel, {\it Towards a theory of arithmetic degrees},  Manuscripta Math.  {\bf 89}  (1996),  no. 4, 427--438. 
%\bibitem[Na]{Na} M. Nagata, Local rings, {\it Interscience New York}, 1962.
%\bibitem[N]{N} E. Noether, {\it Idealtheoric in Ringbereichen, } Math. Annalen  {83} (1921), 24-66.
\bibitem[No]{No} D. G. Northcott, {\it On Irreducible Ideals in Local Rings,} J. London Math. Soc., {32} (1957), 82-88.
%\bibitem[No1]{No1}D.G. Northcott, {\it A note on the coefficients of the abstract Hilbert function}, J. London Math. Soc. {\bf 35} (1960) 209-214.
\bibitem[NR]{NR} D. G. Northcott and D. Rees, {\it Reductions of ideals in local rings}, Proc. Camb, Philos. Soc. {\bf 50} (1954) 145-158.

%\bibitem[O]{O} A. Ooishi, {\it $\Delta$-genera and sectional genera of commutative rings}, Hiroshima  Math. J., {\bf 27} (1987), 361-372.


\bibitem[Sch]{Sch} P. Schenzel, 
{\it On the dimension filtration 
and Cohen-Macaulay filtered modules,} 
Van Oystaeyen, Freddy (ed.), 
Commutative algebra and algebraic geometry, 
New York: Marcel Dekker. 
Lect. Notes Pure Appl. Math., {\bf 206}(1999), 245-264. 


%\bibitem[Sch1]{Sch1} P. Schenzel, On the use of local cohomology in algebra and
%geometry. In: Elias, J. (ed.) et al., Six lectures on commutative
%algebra. Basel (1998), 241-292.

\bibitem[St]{St} R. P. Stanley, {Combinatorics and Commutative Algebra}, {\it Second edition, Birkh\"{a}user Boston}, 1996.
%\bibitem[SV]{SV} J. St\H{u}ckrad and W. Vogel, Buchsbaum rings and applications, {\it Spingger-Veriag, Berlin-Heidelberg-New york}, 1986.
%\bibitem[T]{T} N. V. Trung, {\it Toward a theory of generalized Cohen-Macaulay modules,}  Nagoya Math. J., {\bf 102} (1986), 1-49.

\bibitem[T]{T} N. V. Trung,  \textit{Absolutely superficial sequence}, Math. Proc. Cambrige Phil. Soc,
{\bf 93} (1983), 35-47.

\bibitem[Tr]{Tr} H. L. Truong, {\it Index of reducibility of distinguished parameter ideals and Sequentially Cohen-Macaulay modules}, Proc. Amer. Math. Soc. {\bf 141}, no. 6, 1971-1978.


\bibitem[Tr1]{Tr1} H. L. Truong, {\it Index of reducibility of parameter ideals and Cohen-Macaulay rings}  J. Algebra, {\bf 415}, pp. 35-49.

\bibitem[Tr2]{Tr2} H. L. Truong, {\it The Chern Coefficient and Cohen-Macaulay rings},  preprint.


\bibitem[TY]{TY} H. L. Truong and H. N. Yen, {\it Hilbert functions of socle ideals},  preprint.


%\bibitem[V]{V} W. V. Vasconcelos,  {\it The degrees of graded modules}, Lecture Notes in Summer School on Commutative Algebra, vol. 2, pp 141-196, Centre de Recerca Matematica, Bellaterra (Spain), 1996.
%\bibitem[V1]{V1} W. V. Vasconcelos, {Computational Methods in Commutative algebra and Algebraic Geometry},  {\it Springer  Verlag, Berlin-Heidelderg-New York}, 1998
%\bibitem[V2]{V2} W. V. Vasconcelos, {\it The Chern coefficients of local rings}, Michigan Math. J., {\bf 57} (2008), 725-743.
%\bibitem{G} S. Goto, Approximately Cohen-Macaulay rings, {\it J. Algebra}, {\bf 76} (1982), 214--225.
%\bibitem{GS} S. Goto, H. Sakurai, The equality $I^2 = QI$ in Buchsbaum rings, {\it Rend. Sem. Mat. Univ. Padova} {\bf 110} (2003), 25--56.
%\bibitem{HRS1} W. Heinzer, L.J. Ratliff, K. Shah, On the embedded primary components of ideals I, {\it J. Algebra}, {\bf 167} (1994), 724--744.
%\bibitem{HRS2} W. Heinzer, L.J. Ratliff, K. Shah, On the embedded primary components of ideals II, {\it J. Pure Appl. Algebra}, {\bf 101} (1995), 139--156.
%\bibitem{HRS3} W. Heinzer, L.J. Ratliff, K. Shah, On the embedded primary components of ideals III, {\it J. Algebra}, {\bf 171} (1995), 272--293.
%\bibitem{HRS4} W. Heinzer, L.J. Ratliff, K. Shah, On the embedded primary components of ideals IV, {\it Trans. Amer. Math. Soc.}, {\bf 347} (1995), 701--708.
%\bibitem{Y} Y. Yao, Primary decomposition: Compatibility, independence and linear growth, {\it Proc. Amer. Math. Soc.}, {\bf 130} (2002), 1629--1637.




\end{thebibliography}
\end{document}